\documentclass{article}
\usepackage{amsfonts,amsmath,amssymb,amscd,amsthm}
\usepackage{amsmath}
\usepackage[utf8]{inputenc}
\usepackage[english]{babel}
\usepackage{pst-node}
\usepackage{tikz-cd}
\usepackage{authblk}
\usepackage{mathrsfs}

\usepackage[left=30mm, top=20mm, right=30mm, bottom=25mm]{geometry} 
\long\def\comment#1\endcomment{}

\long\def\comment#1\endcomment{}
\usepackage[unicode, pdftex]{hyperref}
\usepackage{amsrefs}

\DeclareMathOperator{\spam}{span}

\DeclareMathOperator{\diag}{diag}

\DeclareMathOperator{\Tr}{Tr}

\newtheorem{thm}{Theorem}[section]
\newtheorem{claim}[thm]{Claim}
\newtheorem{cor}[thm]{Corollary}
\newtheorem{lem}[thm]{Lemma}
\newtheorem{prop}[thm]{Proposition}
\newtheorem{defin}[thm]{Definition}

\theoremstyle{remark}\newtheorem{rem}[thm]{Remark}

\title{On uniform Hilbert Schmidt stability of groups}
\author[1]{Danil Akhtiamov}
\author[1]{Alon Dogon}
\affil[1]{Einstein Institute of Mathematics, The Hebrew University, Jerusalem, 9190401, Israel.}

\date{}

\begin{document}

\maketitle
 
\begin{abstract}
A group $\Gamma$ is said to be uniformly HS-stable if any map $\varphi : \Gamma \to U(n)$ that is almost a unitary representation (w.r.t. the Hilbert Schmidt norm) is close to a genuine unitary representation of the same dimension.
We present a complete classification of uniformly HS-stable groups among finitely generated residually finite ones. Necessity of the residual finiteness assumption is discussed. A similar result is shown to hold assuming only amenability.
\end{abstract}

\section{Introduction}
In his paper from 1982, D. Kazhdan proves the following theorem:

\begin{thm}[Kazhdan \cite{Kazh}] \label{thm:Kazh}
Let $0<\epsilon<\frac{1}{200}$, let G be an amenable group, and let $\mathcal{H}$ be a Hilbert space. 
Let $\varphi: G \to \mathcal{U}(\mathcal{H})$ be such that it satisfies
\begin{equation*} \label{eq:epsilon_rep}
\forall g,h \in G, \; \Vert \varphi(g) \varphi(h) - \varphi(gh) \Vert_{op} < \epsilon
\end{equation*}
Then there exists a representation $\pi: G \to \mathcal{U}(\mathcal{H})$ with $\forall g \in G, \; \Vert \varphi(g) - \pi(g) \Vert_{op} < 2\epsilon$
\end{thm}

This theorem is a specific answer to a general question that was proposed by Ulam in \cite{Ulam}, which roughly asks: 
Given a map between groups that is close to being a homomorphism, can it be approximated by a genuine homomorphism? Formally, one 
can interpret this as follows (see \cite{DOT}):

\begin{defin}
Given a group $G$, a metric group $(H,d)$ and $\epsilon>0$, a map $\varphi: G \to H$ is said to be an $\epsilon$-homomorphism if
\begin{equation*}
\forall g,h \in G, \; d(\varphi(g) \varphi(h), \varphi(gh)) < \epsilon 
\end{equation*}
\end{defin}

The question is then:
\begin{defin} \label{defin:unif_stab}
Given a class $\mathscr{G}$ of groups together with a class $\mathscr{ H}$ of metric groups. We say $\mathscr{G}$ is uniformly stable with respect to $\mathscr{H}$ if for any $\delta>0$ there exists $\epsilon>0$ such that for any $G \in \mathscr G$,  any $(H,d) \in \mathscr{H}$ and any $\epsilon$-homomorphism $\varphi : G\to H$,
there is a genuine homomorphism $\pi: G\to H$ such that $d(\varphi(g),\pi(g)) < \delta$ for all $g \in G$. When $\mathscr{G} = \{\Gamma\}$ consists of a single group, we say $\Gamma$ is uniformly stable w.r.t. $\mathscr{H}$.
\end{defin}

Kazhdan answered this question positively for the family $\mathscr{G}$  of all amenable groups, and $\mathscr{H}$ of all unitary groups arising from 
arbitrary Hilbert spaces, equipped with the distance induced by the \emph{operator norm}. This was studied further by Burger, Ozawa and Thom in \cite{BOT}, 
where one of the main results also shows that $SL(n, \mathcal{O})$, where $n \geq 3,\, \mathcal{O}$ is the ring of integers of a number field, is uniformly stable w.r.t. to \emph{finite dimensional} unitary groups, equipped with the operator norm.

In this work we are interested in the same question, but where we replace arbitrary unitary groups with the operator norm 
by finite dimensional unitary groups $U(n)$ with the (normalized) Hilbert Schmidt norm:
The norm $\Vert \cdot \Vert_{HS,n}$ on $M_{n \times n}(\mathbb{C})$ is defined by $\Vert A \Vert_{HS,n} = (\frac{1}{n} Tr(AA^*))^{\frac{1}{2}} =(\frac{1}{n} \sum_{i,j} \vert a_{i,j} \vert^2)^{\frac{1}{2}}$ and induces a bi-invariant metric on $U(n)$ by $d_{HS,n} (U,V) = \Vert U - V \Vert_{HS,n}$.
There is no hope, however, to prove uniform stability for the class of all amenable (or even finite) groups with respect to the Hilbert Schmidt norms, as was mentioned in \cite{GH} and \cite{DOT}. The following theorem is the main result of our work, which completely characterizes \emph{uniformly HS-stable groups} (that is, groups that are uniformly stable w.r.t $\mathscr{H} = \{(U(n), d_{HS,n})\}$) among f.g. residually finite groups. This shows that uniform HS-stability occurs extremely rarely in contrast with the case of the operator norm: 

\begin{thm} \label{thm:main}
A finitely generated residually finite group is uniformly HS-stable if and only if it is virtually abelian. 
\end{thm}

The "only if" direction uses the fact that finitely generated residually finite non virtually abelian groups have irreducible representations of arbitrarily large dimensions. 
As a result, we prove they are not uniformly  HS-stable. This is done by replacing a true irreducible representation of dimension $n$ by its projection to an $(n-1)$-dimensional subspace. In this way we obtain "bad" $\epsilon-$representations which are bounded away from all true representations, while $\epsilon$ tends to $0$.  This construction generalises observations made already in \cite{DOT}

The construction mentioned before shows instability. However, it suggests that a relaxed notion of uniform stability, often called \emph{flexible} uniform stability, might still hold for many groups. 
Indeed, a recent result proved by De Chiffre, Ozawa and Thom does hold  for all amenable groups:

\begin{thm}[De Chiffre, Ozawa and Thom \cite{DOT}] \label{thm:DOT_main}
Let $\Gamma$ be a countable amenable group. For any $\epsilon > 0, n \in \mathbb{N}$ and an $\epsilon$-homomorphism $\varphi:\Gamma \to U(n)$ with respect to $d_{HS,n}$, there exists a unitary representation $\pi:G\to U(m)$ for some $n \leq m < n + 2500\epsilon^{2}n$ and an isometry $U: \mathbb{C}^n \hookrightarrow \mathbb{C}^m$ such that
$$ \forall g \in \Gamma, \; \Vert \varphi(g) - U^* \pi(g) U \Vert_{HS,n} < 161\epsilon $$
\end{thm}

This enables us to prove the main technical novelty of the article, by showing that for an amenable group with irreducible finite dimensional unitary representations of bounded dimensions, 
we can deform the $\varepsilon$-representation into a genuine representation without paying the price of enlarging the dimension.

\begin{thm}\label{thm:amenable_case}
\emph{(1)} Given $d \in \mathbb{N}$, the class $\mathscr{H}$ of amenable groups for which all irreducible finite dimensional unitary representations are of dimension $\leq d$ is uniformly stable with respect to the family $\{(U(n), d_{HS,n})\}$ of unitary groups with the corresponding Hilbert Schmidt norms.

\emph{(2)} Let $\Gamma$ be an amenable group. Then $\Gamma$ is uniformly HS-stable if and only if there exists  $d \in \mathbb{N}$, such that all finite dimensional irreducible unitary representations of $\Gamma$ are of dimension $\leq d$. 

\end{thm}

Given this result, we deduce the "if" direction of Theorem \ref{thm:main}. In fact, for any $d$, we show that the relation between $\delta$  and $\epsilon$ does not depend on the choice of the group as long as it contains an abelian subgroup of index $\leq d$:

\begin{cor} \label{cor:vir_abelian_case}
For any $d \in \mathbb{N}$, the class $\mathscr{G}$ of countable virtually abelian groups with an abelian subgroup of index $\leq d$ is uniformly stable
with respect to the family $\{(U(n), d_{HS,n})\}$ of unitary groups with the corresponding Hilbert Schmidt norms.
\end{cor}

To the best of the authors' knowledge, the last result seems to be new even for the infinite cyclic group $\mathbb{Z}$, and for the collection of all finite abelian groups. It is also interesting to compare with the situation of approximate actions, that is, considering the family of symmetric groups with normalised Hamming metrics $\mathscr{H}:=\{(S_n, d_{Hamm})\}$. Becker and Chapman \cite{BC} showed that $\mathbb{Z}$ is \emph{not} uniformly stable w.r.t. $\mathscr{H}$, where as we show it \emph{is} uniformly stable w.r.t. $\{U(n), d_{HS,n}\}$. This seems to be the first result for which $\mathscr{H}$ and $\{U(n), d_{HS,n}\}$ behave differently.

Becker and Chapman \cite{BC} also prove an analog of Theorem \ref{thm:DOT_main} in the realm of symmetric groups with the Hamming distances.
In their manuscript they explain how these results are relevant for the fields of property testing and quantum information theory. It seems that our results might be interpreted in a similar way, the interested reader can find more about connections between stability of groups and quantum information theory in \cite{VidickBlog}.

\begin{rem}

Theorem \ref{thm:main} is not valid without the assumption of $\Gamma$ being residually finite. For example, by \cite{JushenkoMonod} there exists an infinite finitely generated simple amenable group  $D$.  Such a group is not virtually abelian.  However, every finite dimensional representation of $D$ is trivial (by Malcev's theorem \cite{Malcev}) and so Theorem \ref{thm:amenable_case} implies that $D$ is uniformly HS-stable. 

\end{rem}

{\it Question.} This paper shows that within the classes of (a) residually finite groups and (b) amenable groups, the property of uniform HS-stability is equivalent to the property that all finite dimensional irreducible unitary representations are of bounded dimension. A common generalization of (a) and (b) is the family of sofic groups. It is therefore natural to ask if this characterisation is still valid for sofic or even hyperlinear groups.

\section*{Acknowledgements}
We would like to express deep gratitude to our advisor Alex Lubotzky and to Michael Chapman for the very useful conversations, for their help and for initially pointing us towards the questions addressed here. This work is part of the MSc theses of both authors. It was supported by the European Research Council (ERC) under the European Unions Horizon 2020 research and innovation program (Grant No. 692854).

\section{Large irreducible representations and instability}

Throughout this article, all groups are supposed to be discrete and countable. The goal in this section is to assemble the tools needed to prove the "only if" direction of Theorem \ref{thm:main}. We will prove an \emph{instability} result, which generalizes remarks made in \cite{DOT}, \cite{GH}. But first, we need two known lemmas. The following can be found for example as Lemma 6.1. in \cite{GH}.

\begin{lem} \label{lem:HS_prop}
The (normalised) Hilbert Schmidt norm satisfies:

\emph{(1)} (Unitary Invariance) For any $A \in M_n(\mathbb{C}),\; U,V \in U(n)$, we have 
$\Vert UAV \Vert_{HS,n} = \Vert A \Vert_{HS,n}$

\emph{(2)} For any $n,m \in \mathbb{N}$, $A \in M_{n \times m}(\mathbb{C}), B \in M_{m \times m} (\mathbb{C})$, $C \in M_{m \times n}(\mathbb{C})$  we have $\Vert ABC \Vert_{HS,n} \leq \sqrt{\frac{m}{n}}\Vert A \Vert_{op} \Vert B \Vert_{HS,m} \Vert C \Vert_{op}$
\end{lem}

The next lemma can be interpreted as a stability result for the relation defining unitary matrices. A very similar result can be found for example as Lemma 2.8 in \cite{SV}. There is a slight difference in formulation, so we include a proof for completeness. 

\begin{lem} \label{lem:unitary_stab}
Let $M \in M_n(\mathbb{C})$. Then there exists a unitary $R \in U(n)$ such that $\Vert M - R \Vert_{HS,n} \leq \Vert M^* M - I_n \Vert_{HS,n}$
\end{lem}

\begin{proof}
Write the Singular Value Decomposition 
$M = S \Sigma V^*$ where $V,S \in U(n)$ and $\Sigma = \diag(\lambda_1, \dots, \lambda_n)$ 
is a diagonal matrix with the singular values $\lambda_i \geq 0$. 
We have: 

$$\Vert M^* M - I_n \Vert_{HS,n} = \Vert V \Sigma S^* S \Sigma V^* - I_n\Vert_{HS,n} = \Vert V \Sigma^2 V^* - I_n \Vert_{HS,n} = \Vert \Sigma^2 - I_n \Vert_{HS,n}$$

Where we used unitary invariance (\ref{lem:HS_prop}, (1)) in the last equality. Now: $\Vert \Sigma^2 - I_n \Vert_{HS,n}^2 = \frac{1}{n} \sum_{i=1}^n (\lambda_i^2  - 1)^2$.
We always have for $\lambda \geq 0$, $\vert \lambda ^2 - 1 \vert = \vert(\lambda - 1)\vert  \vert(\lambda + 1)\vert \geq \vert(\lambda - 1)\vert$
So we also know that 
$$\Vert \Sigma - I_n \Vert_{HS,n}^2 = \frac{1}{n} \sum_{i=1}^n (\lambda_i - 1)^2 \leq \frac{1}{n} \sum_{i=1}^n (\lambda_i^2  - 1)^2 = \Vert M^* M - I_n\Vert_{HS,n}^2  $$

Now, define the unitary $R = SV^* \in U(n)$. Observe:
$$ \Vert M - R \Vert_{HS,n} = \Vert S \Sigma V^* - SV^* \Vert_{HS,n} = \Vert S (\Sigma - I_n) V^* \Vert_{HS,n} = \Vert \Sigma - I_n \Vert_{HS,n}  \leq \Vert M^* M - I_n\Vert_{HS,n}$$ 
Where in the last equality, we again used the fact that the Hilbert Schmidt norm is unitarily invariant (\ref{lem:HS_prop}, (1)).
\end{proof}

\begin{prop} \label{prop:instab}
If $\Gamma$ is a group with irreducible finite dimensional unitary representations of arbitrarily large dimensions, then $\Gamma$ is \emph{not} uniformly HS-stable.
\end{prop}
\begin{proof}
Let $36 \leq n \in \mathbb{N}$ be such that there exists an irreducible unitary representation $\pi : \Gamma \to U(n)$. Denote by $\mathcal{M} = M_n(\mathbb{C})$ the algebra of $n \times n$ matrices.
Let $P$ be the orthogonal projection operator onto the first $n-1$ coordinates, so that $P$ has rank $n-1$. We will work with the "corner" $P\mathcal{M}P$, which consists of matrices in $\mathcal{M}$ whose last column and row consist of zeros. 
As such, we have a $*$ algebra isomorphism $F: M_{n-1}(\mathbb{C}) \xrightarrow{\sim} P\mathcal{M}P \subset \mathcal{M}$ defined by $F(A) = QAQ^*$ where $Q:\mathbb{C}^{n-1} \hookrightarrow \mathbb{C}^n$ is defined by $Q(e_i) = e_i$ for $i \leq n-1$, $\{e_i\}_{i=1}^n$ the standard basis. Visually:
$$
A \in M_{n-1}(\mathbb{C}) \mapsto
\begin{pmatrix}
A & 0 \\
0 & 0
\end{pmatrix}
\in M_n(\mathbb{C})
$$
Further, we have: $\Vert A \Vert_{HS,n-1} = c \Vert F(A) \Vert_{HS,n}$ for $A \in M_{n-1}(\mathbb{C})$ where $c=\sqrt{\frac{n}{n-1}}$.
Similarly, we can also define $F: Hom_{\mathbb{C}}(\mathbb{C}^n, \mathbb{C}^{n-1}) \xrightarrow{\sim} P\mathcal{M}$ by $F(B) = QB$. We then get the following "functoriality": for $A \in M_{n-1}(\mathbb{C}), \; B \in Hom_{\mathbb{C}}(\mathbb{C}^n, \mathbb{C}^{n-1}), C \in \mathcal{M}$, $F(AB)=F(A)F(B),\;F(BC)=F(B)C$.

Define the map $\varphi : \Gamma \to \mathcal{M}$ by $\varphi(g) = P \pi(g) P$, that is, we simply look at the restriction of $\pi(g)$ to the left upper $(n-1) \times (n-1)$ corner, and surround it by zeros. By definition, $\varphi$ maps to $P \mathcal{M} P$. Let $\varphi' = F^{-1} \circ \varphi: \Gamma \to M_{n-1}(\mathbb{C})$.
We claim $\varphi'$ is a $\frac{c}{\sqrt{n}}$-homomorphism (although technically $\varphi'$ is not valued in $U(n-1)$ yet). Indeed, we have $\Vert I_n - P \Vert_{HS,n}^2 = \frac{1}{n}$, so by (\ref{lem:HS_prop}, (1)) and (\ref{lem:HS_prop}, (2)) and the fact that $\Vert P \Vert_{op} = 1$:

\begin{align*}
\Vert \varphi'(g)\varphi'(h) - \varphi'(gh) \Vert_{HS,n-1} & = c \Vert \varphi(g)\varphi(h) - \varphi(gh) \Vert_{HS,n} = c\Vert P\pi(g) P \pi(h) P - P\pi(gh)P \Vert_{HS,n} \\
& \leq c \Vert \pi(g) P \pi(h) - \pi(g)\pi(h) \Vert_{HS,n} = c\Vert P - I_n \Vert_{HS,n} \leq \frac{c}{\sqrt{n}}
\end{align*}

Since $F(I_{n-1}) = P$, by (\ref{lem:HS_prop}, (1)) and (\ref{lem:HS_prop}, (2)):

\begin{align*}
\Vert I_{n-1} - (\varphi'(g))^* \varphi'(g) \Vert_{HS,n-1} & = c\Vert P - P \pi(g)^* P \pi(g) P \Vert_{HS,n} = c\Vert PI_nP- P \pi(g)^* P \pi(g) P \Vert_{HS,n}\\
& \leq c\Vert I_n - \pi(g)^* P \pi(g) \Vert_{HS,n} = c\Vert I_n - P \Vert_{HS,n} \leq \frac{c}{\sqrt{n}}
\end{align*}

Using Lemma \ref{lem:unitary_stab} we can find unitaries $\psi'(g)$ in $U(n-1)$ such that $\Vert \varphi'(g) - \psi'(g) \Vert_{HS,n-1} \leq \frac{c}{\sqrt{n}}$. Define $\psi = F \circ \psi': \Gamma \to \mathcal{U}(P\mathcal{M}P)$.
Since $\Vert \varphi'(g) \Vert_{op} \leq \Vert Q \Vert_{op} \Vert P \Vert_{op} \Vert \pi(g) \Vert_{op} \Vert P \Vert_{op} \Vert Q^* \Vert_{op} \leq 1,\; \Vert \psi'(g) \Vert_{op} = 1$, we get:
\begin{align*}
\Vert \psi'(g)\psi'(h) - \psi'(gh) \Vert_{HS,n-1} & \leq \Vert \psi'(g)\psi'(h) - \varphi'(gh) \Vert_{HS,n-1} + \frac{c}{\sqrt{n}} \\
& \leq \Vert \psi'(g)\psi'(h) - \varphi'(g)\varphi'(h) \Vert_{HS,n-1} + \frac{2c}{\sqrt{n}} \\
& \leq \Vert (\psi'(g) - \varphi'(g))\psi'(h) \Vert_{HS,n-1} + \Vert \varphi'(g)(\psi'(h) - \varphi'(h)) \Vert_{HS,n-1} +\frac{2c}{\sqrt{n}} \\
& \leq \frac{4c}{\sqrt{n}}
\end{align*}
Thus, $\psi': \Gamma \to U(n-1)$ is a $\frac{4}{\sqrt{n-1}}$-homomorphism.

We now claim that it is bounded away from all true unitary representations of dimension $n-1$ by $\frac{1}{2}$. Assume $\rho': \Gamma \to U(n-1)$ is a representation such that $\Vert \rho'(g) - \psi'(g) \Vert_{HS,n-1} < \frac{1}{2}$ for all $g\in\Gamma$. Define $\rho = F\circ\rho': \Gamma \to \mathcal{U}(P\mathcal{M}P)$.
We consider the bounded family of matrices $\{\rho(g) \pi(g)^*\}_{g\in\Gamma} \subset P\mathcal{M}$, and take its closed convex hull to be $\mathcal{C} \subset P\mathcal{M}$. Recall we had $\Vert \varphi(g) - \psi(g) \Vert_{HS,n} \leq \frac{1}{\sqrt{n}}, \; \Vert \rho(g) - \psi(g) \Vert_{HS,n} < \frac{1}{2c}$. By Lemma \ref{lem:HS_prop} again, we have for all $g$:

\begin{align*}
 \Vert \pi(g) - P \pi(g) P \Vert_{HS,n} &\leq \Vert \pi(g) - \pi(g) P \Vert_{HS,n} + \Vert (I_n - P) \pi(g) P \Vert_{HS,n} \leq 2 \Vert I_n - P \Vert_{HS,n} = \frac{2}{\sqrt{n}} \\
\Vert I_n - \rho(g) \pi(g)^* \Vert_{HS,n} & \leq \Vert I_n - \varphi(g) \pi(g)^* \Vert_{HS,n} + \Vert (\varphi(g) - \rho(g))\pi(g)^* \Vert_{HS,n} \\
& \leq \Vert I_n - \varphi(g) \pi(g)^* \Vert_{HS,n} + \Vert \varphi(g) - \psi(g) \Vert_{HS,n} + \Vert \rho(g) - \psi(g) \Vert_{HS,n} \\
& < \Vert I_n - P \pi(g) P \pi(g)^* \Vert_{HS,n} + \frac{1}{\sqrt{n}} + \frac{1}{2c} \leq \frac{3}{\sqrt{n}} + \frac{1}{2} 
\end{align*}

Notice we used $c>1$. And so $\mathcal{C}$ is contained in the ball $B_{\frac{1}{2} + \frac{3}{\sqrt{n}}}(I_n)$ as we just showed $\{\rho(g) \pi(g)^*\}_{g\in\Gamma}$ is contained in it. Since $\frac{1}{2} + \frac{3}{\sqrt{n}} \leq 1$, $\mathcal{C}$ does not contain $0$. 
Notice that $\Gamma$ acts isometrically on the normed space $P \mathcal{M}$ (equipped with the Hilbert Schmidt norm) by $g \cdot A = \rho(g) A \pi(g)^*$.
This is indeed an action since $\rho = F \circ \rho'$ is a homomorphism, and it preserves the $\Vert \cdot \Vert_{HS,n}$ norm by unitary invariance (\ref{lem:HS_prop}, (1)):
For $A \in P\mathcal{M}$ we have $g \cdot A = (\rho(g) + (I_n - P)) A \pi(g)^*$ where $\rho(g) + (I_n - P)$ is unitary since $\rho(g) \in \mathcal{U}(P\mathcal{M}P)$.
Notice $\mathcal{C}$ is invariant under the action of $\Gamma$, as $\{\rho(g)\pi(g)^*\}_{g \in \Gamma}$ is. Since $\mathcal{C}$ is closed and convex in a Hilbert space, it contains a unique vector of minimum norm $A$.
Since the action is isometric, this operator $A$ is invariant under the $\Gamma$ action, so it satisfies that $\rho(g)A=A\pi(g)$  for all $g$ in $\Gamma$. By "functoriality": $\rho'(g)F^{-1}(A) = F^{-1}(A) \pi(g)$, So $F^{-1}(A):\mathbb{C}^n \to \mathbb{C}^{n-1}$ intertwines $\pi$ and $\rho'$. But since it is non zero and $\pi$ is irreducible, by Schur's Lemma it has to be injective. Which is a contradiction to the existence of such $\rho'$.

Thus, we obtained a family of $\psi_n$ of $\frac{4}{\sqrt{n-1}}$-homomorphisms for arbitrarily large $n$'s, that are bounded away by $\frac{1}{2}$ from any true $n-1$ dimensional representation. So, $\Gamma$ is not uniformly HS-stable.

\end{proof}

The following is where the residual finiteness and finite generation assumptions come in: In this situation being virtually abelian is equivalent to the property that dimensions of finite dimensional irreducible unitary representations are bounded.
\begin{lem} \label{lem:irrep_dims}

\emph{(1)} (Kaplansky \cite{Kaplansky}) If $\Gamma$ is virtually abelian with an abelian subgroup of index $\leq d$, then each finite dimensional irreducible complex representation of $\Gamma$ has dimension $\leq d$.

\emph{(2)} If $\Gamma$ is residually finite, finitely generated but not virtually abelian, then it has irreducible finite dimensional unitary representations of arbitrarily large dimensions.

\end{lem}

\begin{rem}
It should be noted that virtually abelian groups are \emph{exactly} the (discrete) groups with the property that any irreducible unitary representation is finite dimensional. This follows from works of Thoma \cite{Thoma} and Glimm \cite{Glimm}, see also \cite{Moore}. With that said, we will not be using this fact.
\end{rem}

\begin{proof}

(1) Let $A \leq \Gamma$ be an abelian subgroup of index $k, k \leq d$. Let $\rho: \Gamma \to GL(V)$ be a finite dimensional irreducible $\Gamma$-representation. Consider $\rho\vert_{A}$. This is as an $A$-representation and therefore it has an irreducible $A$ sub-representation (since $V$ is finite dimensional, it has an $A$- invariant subspace of minimal dimension).
Since $A$ is abelian, this irreducible subspace is spanned by a single vector $v_0$ and acted on by $A$ with some character $\chi$. Let $g_1, \dots, g_k$ denote some left coset representatives of $A$ in $\Gamma$.
Since $V$ is irreducible, we have $\spam(\Gamma \cdot v_0) = V$. Since $A$ preserves $\spam(v_0)$, we obtain: 
$$\spam(\Gamma \cdot v_0) = \spam\big(\bigcup_{i=1}^{k} g_i A \cdot v_0 \big) = \spam\{g_i\cdot v_0\}_{i=1}^{k}$$
Thus, $\dim(V) = \dim(\spam\{g_i\cdot v_0\}_{i=1}^{k}) \leq d$.

(2) Assume all finite dimensional irreducible unitary representations of $\Gamma$ (f.g. residually finite) have dimension less than $d$, we shall show it is virtually abelian.
For each $e \neq x \in \Gamma$, there exists a finite index normal subgroup $N_x \triangleleft \Gamma$ with $x\notin N_x\triangleleft\Gamma$.
Since irreducible finite dimensional representations of a finite group separate points, there exists a (unitary) irrep $\rho_{x}:\Gamma/N_x\to U(n)$
for some $n\in\mathbb{N}$ with $\rho_{x}(x\cdot N_x) \neq I_n$. If we pre-compose $\rho_{x}$ with the quotient map $\Gamma \twoheadrightarrow \Gamma / N_x$, we get an irreducible representation of $\Gamma$, which we will still denote by $\rho_{x}$. 
Thus by assumption: $n \leq d$, and $\rho_x(x) \neq I_n$. Since $\rho_x(\Gamma)$ are finite, by the Jordan-Schur Theorem (see for example \cite{TerryBlog}) there is an integer $C$ depending only on $d$ s.t. for all $x \neq e$, there is an abelian subgroup $A'_{x} \leq \rho_{x}(\Gamma)$ with index bounded by $C$.
Pull back $A_{x}=\rho_{x}^{-1}(A'_{x})$, we have $[\Gamma:A_{x}]\leq C$ by the correspondence theorem. Let $A=\bigcap_{x} A_{x} \leq \Gamma$.
This intersection is actually of finite index: it is known that $\Gamma$, being finitely generated, has finitely many subgroups of index less than $C$. And so, $A$ is a finite intersection of finite index subgroups and therefore is of finite index. Lastly, we show $A$ is abelian. 
Indeed, if $g,h\in A$ do not commute, then the commutator $[g,h] \ne 1$ satisfies $\rho_{x}([g,h])=1$ for all $e \neq x \in \Gamma$ because $\rho_x(g), \rho_x(h)$ belong to the abelian subgroup $A_x$. On the one hand $[g,h] \neq e$ in particular implies $\rho_{[g,h]} ([g,h])=1$. On the other hand,  $\rho_{[g,h]} ([g,h]) \ne 1$ by construction, which gives us a contradiction. Thus, $A$ is abelian.

\end{proof}

\section{Uniform stability of virtually abelian groups and beyond}

In this section we set out to prove Theorem \ref{thm:amenable_case}, which will give the "if" direction of our main result \ref{thm:main}.  Recall the following simple fact:

\begin{lem} \label{lem: unireps}

Let $\Gamma$ be any group and let $\pi: \Gamma \to U(n)$ be a unitary representation. Then it is completely reducible. That is, there exist  $W \in U(n)$ and irreducible representations $\pi_1: \Gamma \to U(n_1), \dots, \pi_k: \Gamma \to U(n_k)$ with $n_1+ \dots +n_k=n$ such that $\pi=W^* \diag(\pi_1, \dots, \pi_k)W$. 

\end{lem}

\begin{proof}[Proof of Theorem \ref{thm:amenable_case}]

The "only if" part of (2) follows directly from Proposition \ref{prop:instab}. The converse direction of (2) clearly follows from (1).

For showing (1),  we will prove that for any $\Gamma$ in the class $\mathscr{H}$, any $\epsilon$-homomorphism $\varphi:\Gamma\to U(n)$ into the unitary group $U(n)$ with the Hilbert Schmidt norm $\Vert\cdot\Vert_{HS,n}$ and for
$$\delta=161\epsilon + 100\big( (\sqrt{2}+1)\sqrt{1+2500\epsilon^2}+\sqrt{d-1} \big)\epsilon$$ there exists a true representation $\rho: \Gamma \to U(n)$ with $\Vert \varphi(g) - \rho(g) \Vert_{HS,n} < \delta$ for all $g \in \Gamma$.

Firstly, by De Chiffre, Ozawa and Thom's result (Theorem \ref{thm:DOT_main}), there exists $n \leq m < n+2500\epsilon^{2}n$, an isometry $U:\mathbb{C}^{n}\hookrightarrow\mathbb{C}^{m}$ and a unitary representation $\pi:\Gamma\to U(m)$ such that:

\begin{equation} \label{eq:DOT}
\forall g\in\Gamma\text{, }\Vert\varphi(g)-U^{*}\pi(g)U\Vert_{HS,n}<161\epsilon
\end{equation}

If $\epsilon<\frac{1}{50\sqrt{n}}$, i.e. if $n+2500\epsilon^{2}n<n+1$, then $m=n$ automatically and we have nothing to prove ($\rho(g)= U^* \pi(g) U$ is as required). So, from now one we assume that $ \frac{1}{n} \leq 2500\epsilon^2$.

Applying Lemma \ref{lem: unireps} to $\pi$, we deduce that there exists $W \in U(m)$, $k \leq m$ and irreducible representations $\pi_1, ..., \pi_k$, such that $\pi(g)=W^{*}\diag(\pi_{1}(g),\dots,\pi_{k}(g))W$ for every $g \in \Gamma$. Note that $d_i:=dim(\pi_i) \le d$ for each $i=1 \dots k$ by assumption. 
Replace $U$ by $WU$ and $\pi(g)$ by $\diag(\pi_{1}(g),\dots,\pi_{k}(g))$. Note that the new $U$ is still an isometry with the same domain and codomain and we still have $\Vert\varphi(g)-U^{*}\pi(g)U\Vert_{HS,n}<161\epsilon$. Denote the standard basis by $\{e_1,...,e_m\}$. Notice $\pi(g)$ is now block-diagonal in this basis.

Let $P = UU^*$ be the projection onto the image of $U$. Denote by $[P]_n$ the $m \times n$ matrix of the first $n$  columns of $P$ in the basis $\{e_1,...,e_m\}$.  That is, $[P]_n: \mathbb{C}^n \to \mathbb{C}^m$ is given by $[P]_ne_i=Pe_i$ for $i=1 \dots n$. Define $M:= U^* [P]_n \in M_n(\mathbb{C})$. Notice $U^*P = U^* UU^* = U^*$, so this is just the square matrix consisting of the columns $\{U^* e_1, U^* e_2, \dots, U^* e_n\}$.

\begin{claim} \label{claim:PM}

\emph{(1)} $\Vert I_m - P \Vert_{HS,m}^2 =\Vert I_m - P^* P \Vert_{HS,m}^2  \leq 2500 \epsilon^2$ 

\emph{(2)} $\Vert [P]_n^* [P]_n - I_n \Vert_{HS,n} \leq  \sqrt{1+2500\epsilon^2} 50 \epsilon$

\emph{(3)} $\Vert M^* M - I_n\Vert_{HS,n} \leq \sqrt{1+2500\epsilon^2} 50 \epsilon$

\end{claim}

\begin{proof}

(1) Since $P$ is a rank $n$ projection, $I_m - P$ is a rank $m - n$ projection, we have that 
$$\Vert I_m - P \Vert_{HS,m}^2 = \frac{1}{m} \Tr((I_m - P)^* (I_m - P)) = \frac{1}{m} \Tr((I_m - P)) = \frac{m - n}{m} \leq 2500 \epsilon^2$$

(2)  Notice that $[P]_n^* [P]_n$ appears as a corner of $P^* P$ as follows:

$$
P^*P =
\begin{pmatrix}
[P]_n^* [P]_n & * \\
* & *
\end{pmatrix}, \; 
P^*P - I_m =
\begin{pmatrix}
[P]_n^* [P]_n - I_n & * \\
* & *
\end{pmatrix}
$$

Thus, since $n\Vert [P]_n^* [P]_n - I_n \Vert_{HS,n}^2$ is the sum of square entries of $[P]_n^* [P]_n - I_n$, we have that:
$$\frac{n}{m}\Vert [P]_n^* [P]_n - I_n \Vert_{HS,n}^2 \leq \Vert P^* P - I_m \Vert_{HS,m}^2 \leq 2500 \epsilon^2$$
So:
$$\Vert [P]_n^* [P]_n - I_n \Vert_{HS,n} \leq \sqrt{m/n} 50\epsilon \leq \sqrt{1+2500\epsilon^2} 50 \epsilon$$

(3) We just pull (2) back to $\mathbb{C}^n$, Since $P = UU^*$ is a projection:
$$\Vert M^* M - I_n\Vert_{HS,n} = \Vert [P]_n^* U U^* [P]_n - I_n\Vert_{HS,n} = \Vert [P]_n^* [P]_n - I_n\Vert_{HS,n} \leq
\sqrt{1+2500\epsilon^2} 50\epsilon $$

\end{proof}

Hence, we can apply Lemma \ref{lem:unitary_stab} to (3) in claim \ref{claim:PM} and obtain a unitary $R \in U(n)$ which is $\sqrt{1+2500\epsilon^2} 50 \epsilon$-close to $M$ in the $HS$-norm.

Define the block-diagonal matrix $\pi'(g):=\diag(\pi_1(g), \dots, \pi_r(g), 1, \dots, 1)$ for the maximal $r \leq k$ such that $d_1+ d_2 +\dots+ d_r \leq n$ and fill in the rest of the diagonal by $1$'s so that $\pi'(g)$ is an $n \times n$ matrix. Note that it guarantees $d_1+ d_2 +\dots+ d_k \geq n-d+1$, since all blocks have size $\leq d$. Finally, define the representation $\rho : \Gamma \to U(n)$ as $\rho(g)=R \pi'(g) R^*$. This is a unitary representation since $\pi'(g)$ is a unitary representation and $R$ is unitary.

{\it Remark:} If $d=1$, then each $d_i$ equals $1$ as well and therefore each $\pi_i$ is a character, so let us denote it by $\chi_i$ because this is more natural in this case.  Let $R_i$ be the $i-$th column of $R$. Using that $R$ is unitary, i.e. $R^*R=Id$, we obtain $\rho(g)R_i=R \pi'(g) (R^*R_i)=R \pi'(g) e_i=R \chi_i(g) e_i=\chi_i(g) (Re_i)=\chi_i(g) R_i$, so there is a nice short formula for $\rho$ in the orthonormal basis given by $\{R_1,...,R_n\}$. 

\begin{claim} \label{claim:inequalities}

\emph{(1)} $\Vert M\pi'(g)M^*-\rho(g) \Vert_{HS,n} \leq 100\epsilon\sqrt{1+2500\epsilon^2}$

\emph{(2)} $\Vert U^*\pi(g)U-M\pi'(g)M^* \Vert_{HS,n} \leq 100(\sqrt{2 + 5000\epsilon^2} + \sqrt{d-1})\epsilon$

\end{claim}

\begin{proof}

(1) Note that $\Vert \pi'(g) \Vert_{op}=\Vert R \Vert_{op}=1$ because these two operators are unitary, $\Vert M \Vert_{op}= \Vert U^*  [P]_n \Vert_{op} \leq \Vert U^* \Vert_{op} \Vert [P]_n \Vert_{op}  \leq \Vert U^* \Vert_{op} \Vert P \Vert_{op}=1$. Using this, (\ref{lem:HS_prop}, (2))  and   Claim \ref{claim:PM} we obtain:

\begin{align*}
\Vert M\pi'(g)M^*-\rho(g) \Vert_{HS,n} & =  \Vert M\pi'(g)M^*-R \pi'(g)R^* \Vert_{HS,n}      &\\
    & \leq \Vert (M-R)\pi'(g)M^* \Vert_{HS,n}+\Vert R\pi'(g)(M-R)^* \Vert_{HS,n} \\
    & \leq \Vert (M-R)\Vert_{HS,n} \Vert \pi'(g)M^* \Vert_{op}+\Vert R\pi'(g)\Vert_{op} \Vert(M-R)^* \Vert_{HS,n}     &(\ref{lem:HS_prop}, (2))\\
    & \leq 2\Vert (M-R)\Vert_{HS,n} \leq 100\epsilon\sqrt{1+2500\epsilon^2}  
\end{align*}

(2) Let $Q: \spam(e_1,...,e_n) \hookrightarrow \spam(e_1,...,e_m)$ be given by $Q e_i := e_i$ for $i=1 \dots  n$. It gives us $Q^*(e_i) = e_i$ for $i=1 \dots n$ and $Q^*(e_i) = 0$ for $i>n$. Denote $n \geq n':=d_1+\dots+d_k \geq n-d+1$.

 Recall the bound $\Vert I_m - P \Vert_{HS,m}^2 \leq 2500\epsilon^2$ in \ref{claim:PM}. We deduce:
$$\Vert [P]_nQ^*-I_m \Vert_{HS,m}^2=\frac{1}{m} \sum_{i=1}^{n} \Vert Pe_i - e_i \Vert^2+\frac{m-n}{m} \leq \frac{1}{m} \sum_{i=1}^{m} \Vert Pe_i - e_i \Vert^2+\frac{m-n}{n} \leq  2500\epsilon^2 +2500\epsilon^2 $$

So $\Vert [P]_nQ^*-I_m \Vert_{HS,m} \leq 50\sqrt{2}\epsilon$. Note that $\pi'(g)e_i=Q^*\pi(g)Qe_i$ for $i \leq n'$ and $\Vert \pi'(g)e_i-Q^*\pi(g)Qe_i \Vert = \Vert e_i-Q^*\pi(g)e_i \Vert \leq 2$ for $n \geq i>n'$, since $\pi'$ was defined to act trivially on these vectors. This implies:

$$\Vert Q^*\pi(g)Q-\pi'(g) \Vert_{HS,n}^2=\frac{1}{n} \sum_{i=1}^n \Vert Q^*\pi(g)Qe_i-\pi'(g)e_i \Vert^2=\frac{1}{n} \sum_{i=n'+1}^n \Vert Q^*\pi(g)Qe_i-\pi'(g)e_i \Vert^2 \leq \frac{4(n-n')}{n}$$

Now recall we assumed $\frac{1}{n} \leq 2500\epsilon^2$. Since $n - n' \leq d-1$, we deduce $\Vert Q^*\pi(g)Q-\pi'(g) \Vert_{HS,n} \leq 100\sqrt{d-1}\epsilon$.
Using (\ref{lem:HS_prop}, (2)) again and also using $\Vert U \Vert_{op}=\Vert U^* \Vert_{op}=1$, $\Vert [P]_n \Vert_{op}=\Vert [P]_n^* \Vert_{op} \le \Vert P \Vert_{op}= 1$ we obtain:

\begin{align*}
\Vert U^*\pi(g)U & -M\pi'(g)M^* \Vert_{HS,n} = \Vert U^*\pi(g)U-U^*[P]_n\pi'(g)[P]_n^*U \Vert_{HS,n}      &\\
& \leq \sqrt{\frac{m}{n}}\Vert \pi(g)-[P]_n\pi'(g)[P]_n^* \Vert_{HS,m}    &(\ref{lem:HS_prop}, (2)) \\
& \leq \sqrt{\frac{m}{n}}\Vert \pi(g)-[P]_nQ^*\pi(g)Q[P]_n^* \Vert_{HS,m}    \\
& + \sqrt{\frac{m}{n}}\Vert [P]_nQ^*\pi(g)Q[P]_n^* - [P]_n\pi'(g)[P]_n^* \Vert_{HS,m}    \\
& \leq \sqrt{\frac{m}{n}}\Vert (I_m-[P]_nQ^*)\pi(g) \Vert_{HS,m} + \sqrt{\frac{m}{n}}\Vert [P]_nQ^*\pi(g)(I_m-Q[P]_n^*) \Vert_{HS,m}   \\
& +\Vert Q^*\pi(g)Q-\pi'(g) \Vert_{HS,n}    &(\ref{lem:HS_prop}, (2))\\
& \leq 2\sqrt{\frac{m}{n}}\Vert (I_m-[P]_nQ^*) \Vert_{HS,m}+100\sqrt{(d-1)}\epsilon \\
& \leq \sqrt{1 + 2500\epsilon^2}100\sqrt{2}\epsilon + 100\sqrt{(d-1)}\epsilon   &(\ref{lem:HS_prop}, (2))
\end{align*}

Notice that in the second to last line we used the fact that $\Vert (I_m - [P]_nQ^*) \Vert_{HS,m} = \Vert (I_m - Q[P]_n^*) \Vert_{HS,m}$.

\end{proof}

By combining both inequalities of Claim \ref{claim:inequalities}, we obtain:
\begin{align*}
\Vert U^* \pi(g) U - \rho(g) \Vert_{HS,n} & \leq \Vert U^* \pi(g) U - M \pi'(g) M^* \Vert_{HS,n}+\Vert \rho(g)- M \pi'(g) M^*\Vert_{HS,n} \\
& \leq  100(\sqrt{2 + 5000\epsilon^2} + \sqrt{d-1})\epsilon+100\epsilon\sqrt{1+2500\epsilon^2} \\
& = 100((\sqrt{2}+1)\sqrt{1+2500\epsilon^2}+\sqrt{d-1})\epsilon
\end{align*}

Hence, using the inequality guaranteed by De Chiffre, Ozawa, Thom (\ref{eq:DOT}) and the triangular inequality:
\begin{align*}
\Vert \varphi(g) - \rho(g) \Vert_{HS,n} & \leq \Vert U^* \pi(g) U - \varphi(g) \Vert_{HS,n} + \Vert U^* \pi(g) U - \rho(g) \Vert_{HS,n} \\
& < 161\epsilon + 100\big( (\sqrt{2}+1)\sqrt{1+2500\epsilon^2}+\sqrt{d-1} \big)\epsilon
\end{align*}

\end{proof}

As a special case of Theorem \ref{thm:amenable_case}, we obtain uniform HS-stability in the virtually abelian setting:

\begin{proof} [Proof of Corollary \ref{cor:vir_abelian_case}] 

Let $G$ be from the class $\mathscr{G}$ of virtually abelian groups with an abelian subgroup of index $\leq d$. Lemma \ref{lem:irrep_dims}, (1) tells us that the dimensions of finite dimensional irreducible unitary representations of $G$ are bounded by $d$ as well.
Since any virtually abelian group is amenable, we conclude that $\mathscr{G} \subseteq \mathscr{H}$ and therefore Theorem \ref{thm:amenable_case} applies directly. 

\end{proof}

We now deduce Theorem \ref{thm:main} as a direct consequence of the last section and the present one:

\begin{proof}[Proof of Theorem \ref{thm:main}]
Let $\Gamma$ be finitely generated and residually finite. Assume it is virtually abelian. Then in particular, Corollary \ref{cor:vir_abelian_case} shows it is uniformly HS-stable. 
Conversely, assume $\Gamma$ is not virtually abelian. By Lemma \ref{lem:irrep_dims}, (2) we know $\Gamma$ must have irreducible unitary representations of arbitrarily large dimensions. As a result, Proposition \ref{prop:instab} shows $\Gamma$ is not uniformly HS-stable.
\end{proof}


\end{document}